\documentclass[english,12pt,a4paper]{article}
\usepackage[T1]{fontenc}
\usepackage{babel}
\usepackage{amsmath,amsthm,amssymb}
\usepackage{mathtools}
\usepackage{graphicx}
\usepackage{ragged2e}
\usepackage{bigints}
\usepackage{units}
\usepackage{hyperref}
\newtheorem{thm}{Theorem}[section]

\newcommand{\norm}[1]{\left\lVert#1\right\rVert}
\begin{document}

		\title{Controllability of a Class of Nonlinear Networked Systems}
	
	\author{Aleena Thomas, Abhijith Ajayakumar, and Raju K.George
		\thanks{The first and second authors are funded by the University Grants Commission, India and Council for Scientific and Industrial Research, India,  respectively.}
		\thanks{The authors are with the Department of Mathematics, Indian Institute of Space Science and Technology, Thiruvananthapuram, Kerala, India.\url{aleenathomas.22@res.iist.ac.in, abhijithajayakumar.19@res.iist.ac.in, george@iist.ac.in}}}

	\maketitle

		\begin{abstract}
			
			Various controllability conditions have been obtained by researchers for heterogeneous networked systems with linear dynamics. However, the literature for nonlinear, heterogeneous networked systems is comparatively less. In this paper we analyse the controllabiity aspect of a nonlinearly perturbed linear networked system. The basic assumption is that the linear system is controllable and the nonlinear perturbation functions satisfy \textit{H\"older continuity} condition and in particular \textit{Lipschitz condition}. The \textit{Boyd-Wong} fixed point theorem is employed to prove controllability of the nonlinear system. The result is illustrated with numerical examples.
			
		\end{abstract}
		
		\section{Problem Formulation}
		\label{sec3}
		In this work, we consider a networked system with $N$ nodes, each having distinct dimension and following heterogeneous dynamics. Let the dynamics of the $i-$th node be denoted by $(A_i,B_i,C_i)$, where $A_i\in \mathbb{R}^{n_i\times n_i}, B_i\in \mathbb{R}^{p_i\times n_i} $ and $C_i\in \mathbb{R}^{n_i\times m}$ are the state, input and output matrices of node $i$, respectively.
		
		These $N$ nodes are connected in a network topology $L=[\beta_{ij}]_{i,j=1}^{N}$, where $$\beta_{ij}\begin{dcases}
			\neq 0&-\, \text{if there is an edge from node $j$ to node $i$}\\
			=0&-\, \text{otherwise}
		\end{dcases}$$ The matrix of external control inputs is given by $\Delta=[\delta_i]_{i=1}^{N}$, where $$\delta_i=\begin{dcases}
			1& -\, \text{if node $i$ is under external control input}\\
			0& -\, \text{otherwise}
		\end{dcases}$$ Inner-coupling matrices, signifying the inner-interactions among the nodes are denoted by $H_i\in \mathbb{R}^{n_i\times m} $.\\
		
		\par  Consider the case when a nonlinear perturbation $f_i(t,x_i(t))$ is applied to each node, where $f_i$ is a nonlinear function defined from $\mathbb{R}\times \mathbb{R}^{n_i}$ to $\mathbb{R}^{n_i}$. Then, under the connection according to $L$ and with the effect of the nonlinear perturbation, the dynamics of the $i$-th node becomes
		\begin{equation}\label{sys-i}
			\dot{x_i}=A_ix_i+\beta_{ij}H_iC_jx_j+\delta_iB_iu_i+f_i(t,x_i(t))
		\end{equation}
		The compact form representation of the networked system is given by
		\begin{equation}\label{mainsys}
			\dot{X}=\mathcal{A} X+\mathcal{B} U+F(t,X(t))
		\end{equation}
		where \begin{align*}
			\mathcal{A}&=A+(\beta_{ij}H_iC_j)_{i,j=1}^{N}; A=blockdiag\{A_1,\cdots,A_N\}\\
			\mathcal{B}&=blockdiag\{B_1,\cdots,B_N\}\\
			F(t,X(t))&=[f_1(t,x_1(t)),\cdots,f_N(t,x_N(t))]^T\\
			L&=[\beta_{ij}]_{i,j=1}^{N}\\
			\Delta&=diag\{\delta_1,\cdots,\delta_N\}
		\end{align*}
		The problem under consideration is the study of controllability of the network (\ref{mainsys}).
		
		\section{Main Result}
		\label{sec4}
		In this section, we will establish a sufficient condition for the controllability of the nonlinear networked system (\ref{mainsys}) based on the fixed point theorem by \textit{Boyd \& Wong}.\\
		
		By Theorem, the \textbf{solution map} of the networked system (\ref{mainsys}) is 
		\begin{align}
			\mathcal{K}&:\mathcal{L}^2[t_0,t_1]\to \mathcal{L}^2[t_0,t_1]\nonumber\\
			(\mathcal{K}x)(t)&\coloneqq \Phi(t,t_0)x_0+\int_{t_0}^{t}\Phi(t,\tau)\bigl[\Psi\tilde{u}(x,\tau)+F(\tau,x(\tau))\bigr]d\tau\label{fp}
		\end{align}
		where \begin{equation}
			\tilde{u}(x,t)=\Psi^T\Phi^T(t_1,t)\mathbb{W}^{-1}\biggl[x_1-\Phi(t_1,t_0)x_0-\int_{t_0}^{t_1}\Phi(t_1,\tau)F(\tau,x(\tau))d\tau\biggr]
		\end{equation} 
		Any fixed point of $\mathcal{K}$ will be a solution to the control system \eqref{mainsys} satifying the initial condition $x(t_0)=x_0$ and the desired final condition $x(t_1)=x_1$. Hence, inorder to show that the system is controllable, it is enough to show that the map $\mathcal{K}$ has a fixed point. Let
		\begin{align*}
			||\Phi(t_1,t_0)||&\leq \alpha_0\\
			||\Psi\Psi^T||&\leq \beta\\
			||\Phi^T(t_1,t)||&\leq \gamma\\
			||\mathbb{W}^{-1}||&\leq \delta
		\end{align*}
		
		\begin{thm}\label{thm1}
			Suppose that the linear part $(\mathcal{A}, \mathcal{B})$ of the nonlinear networked system (\ref{mainsys}) is controllable. Further, let the nonlinear perturbation $F(t,X(t))$ be H\"older continuous with respect to $X$ with exponent $\rho$ and multiplicative constant $\alpha$. If $M=\alpha\alpha_0^2\beta\gamma\delta(t_1-t_0)+\alpha\alpha_0$ and $Mt^\rho<t, \forall t>0$, then the nonlinear networked system is controllable.
		\end{thm}
		
		\begin{proof}
			Using \eqref{fp},
			
			$$\norm{(\mathcal{K}X)(t)-(\mathcal{K}Y)(t)}=\norm{\int_{t_0}^{t}\Phi(t,\tau)\big[\Psi(\tilde{u}(X,\tau)-\tilde{u}(Y,\tau))+(F(\tau,X(\tau))-F(\tau,Y(\tau)))\big]d\tau}$$
			\begin{align*}
				&\leq \norm{\int_{t_0}^{t}\Phi(t,\tau)\big[\Psi\big(\tilde{u}(X,\tau)-\tilde{u}(Y,\tau)\big)\big]\bigg|\bigg|+\bigg|\bigg|\int_{t_0}^{t}\Phi(t,\tau)\big[F(\tau,X(\tau))-F(\tau,Y(\tau))\big]d\tau}\\
				&\leq \norm{\bigintsss_{t_0}^{t}\Phi(t,\tau)\Psi\Psi^T\Phi^T(t_1,\tau)W^{-1}\int_{t_0}^{t_1}\Phi(t_1,s)[F(s,X(s))-F(s,Y(s))]ds d\tau}\\
				&\qquad\qquad\qquad+\norm{\int_{t_0}^{t}\Phi(t,\tau)[F(\tau,X(\tau))-F(\tau,Y(\tau))]d\tau}\\
				&\leq \alpha\alpha_0^2\beta\gamma\delta\big[	\parallel X-Y\parallel^\rho(t_1-t_0)(t-t_0)\big]+\alpha\alpha_0	\parallel X-Y	\parallel^\rho(t-t_0)
			\end{align*}
			That is, \begin{equation}
				\parallel \mathcal{K}X-\mathcal{K}Y	\parallel\leq \big[\alpha\alpha_0^2\beta\gamma\delta(t_1-t_0)+\alpha\alpha_0\big]\norm{X-Y}^\rho
			\end{equation}
			\begin{equation}\label{ineq}
				\parallel\mathcal{K}X-\mathcal{K}Y	\parallel\leq M	\parallel X-Y	\parallel ^\rho
			\end{equation}
			Define \begin{align*}
				\chi&:\mathbb{R}^+\to \mathbb{R}^+\\
				\chi(t)&=Mt^\rho
			\end{align*}
			Then from \eqref{ineq} , $$d(\mathcal{K}X,\mathcal{K}Y)\leq \chi(d(X,Y))$$ where $d$ is the metric induced by the norm.\\
			Also, note that the function $\chi$ is upper-semi continuous. Further, by the choice of $M$ and $\rho$, $\mathcal{K}$ has a unique fixed point, byBoyd-Wong Fixed Point theorem. Hence, system (\ref{mainsys}) is controllable.
			
		\end{proof}
		
		\section{Conclusion}
		\label{sec5}
		In this paper we have analysed the controllability of a networked system with nonlinearities added to each node. Distinct node dimensions and inner-coupling matrices make the system more generalized. A sufficient condition for controllability of the networked system is obtained by the \textit{Boyd \& Wong} fixed point theorem. The examples provide insights into the intricacies of the condition. The result is particularly useful in the case of sub-linear and Lipschitzian perturbations applied to the nodes. We expect to have more sufficient conditions for controllability of nonlinear networked systems. Study of networked systems where the connections are nonlinear is also a possible research direction.
		
		\section*{Acknowledgment}
		The first and second authors thank the University Grants Commission, India and  Indian Institue of Space Science and Technology, India for the financial support, respectively. We thank the Department of Mathematics, Indian Institue of Space Science and Technology, Thiruvananthapuram, Kerala, India for providing all the necessary facilities to pursue this research work.

\end{document}